\documentclass{amsart}
\usepackage{graphicx}
\usepackage{latexsym}
\usepackage{amsfonts}
\usepackage[all]{xy}
\usepackage{amssymb, mathrsfs, amsfonts, amsmath}
\usepackage{amsbsy}
\usepackage{amsfonts}
\newtheorem{theorem}{Theorem}
\newtheorem{example}{Example}
\numberwithin{equation}{section}

\begin{document}
	
	\title[On the construction of complete expanding gradient Ricci solitons]{On the construction of complete expanding gradient Ricci solitons}
	
	\author{M\'arcio de Sousa}
	
	\author{Paula Bonfim}
	
	\author{Romildo Pina}
	
	\author{Tib\'erio Martins}
	
	\address{ICET - CUA, Universidade Federal de Mato Grosso - Av. Universit\'aria nº 3.500, Pontal do Araguaia, MT, Brazil.}
	
	\address{IME, Universidade Federal de Goi\'as - Caixa Postal 131, 74001-970, Goi\^ania, GO, Brazil.}

	\email{marciolemesew@yahoo.com.br}
	\email{paulacorreiacatu@hotmail.com}
	\email{romildo@ufg.br}
	\email{tiberio.b@gmail.com}
	
	\date{\today}

\begin{abstract}
\noindent
We study gradient Ricci solitons warped products whose base is the Euclidean space. We show that the warping functions of these manifolds are invariant under the $(n-1)$-dimensional translation group. We characterize the potential function when the torsion function depends only on one variable, which is a particular case of invariance by translation. From this study, we derive complete examples of expanding gradient Ricci solitons.
\end{abstract}

\maketitle
\section{Introduction}

Gradient Ricci solitons are natural generalizations of Einstein manifolds, and their importance  attracts the attention of many researchers. These manifolds are self-similar solutions of the Ricci flow $\partial g(t)/\partial t = -2Ric_{g(t)}$ and appear as limits of expansion of singularities in this flow.

Let $(M,g)$ be a Riemannian manifold of dimension $n\geq 3$. We say that $(M,g)$ is a gradient Ricci soliton if there exists a differentiable function $h : M \rightarrow \mathbb{R}$ such that 

\begin{equation} \label{eqGRS}
Ric_g + Hess_g(h) = \rho g.
\end{equation}
where $Ric_g$ is the Ricci tensor, $Hess_g(h)$ is the Hessian of $h$ with respect to the metric $g$, and $\rho\in\mathbb{R}$. A gradient Ricci soliton is said to be shrinking, steady, or expanding if $\rho>0$, $\rho=0$, or $\rho<0$, respectively. When $h$ is a constant function, we have an Einstein manifold.

Examples and classifications of gradient Ricci solitons are found in the literature over the approximately thirty years since the introduction of the Ricci solitons concept in \cite{Hamilton}. Most of the examples found are classified as shrinking or steady.

Considering $\mathbb{R}^n$ with the canonical metric $g_0$, $(\mathbb{R}^n,g_0)$ is a gradient Ricci soliton with a potential function given by $h(x)=A|x|^2/2+g_0(x,B)+C$, with $A,C\in\mathbb{R}$ and $B\in\mathbb{R}^n$. They are called Gaussian solitons. By varying the value of the constant $A$, these can be shrinking, steady, or expanding.

A characterization for conformally flat complete shrinking gradient Ricci solitons was given in \cite{Garcia-Rio}, which shows that such manifolds are essentially quotients of $\mathbb{R}^n$, $\mathbb{S}^n$, or $\mathbb{R}\times\mathbb{S}^{n-1}$.

Robert Bryant proved in \cite{Bryant} that there is a single complete, steady, gradient Ricci soliton that is spherically symmetric for any $n\geq 3$, which is known as the Bryant soliton.
In 2012, Cao and Chen showed in \cite{Cao} that any complete, locally conformally flat steady gradient Ricci soliton is either flat or isometric to the Bryant soliton.

In \cite{Catino1} and \cite{YangZhang}, the authors proved several classification theorems for gradient expanding and steady Ricci solitons. In \cite{Catino1}, they proved that the only complete expanding solitons with nonnegative sectional curvature and integrable scalar curvature are quotients of the Gaussian soliton. In \cite{YangZhang}, they proved that a complete noncompact radially Ricci flat gradient expanding Ricci soliton with nonnegative Ricci curvature is a finite quotient of $R^n$. Moreover, they proved that a complete noncompact gradient expanding Ricci soliton with  $Ric \geq 0 $ and $ Div^ 4 Rm = 0 $  is a finite quotient of $R^n$.

It is important to remember that a Riemannian warped product is a product $B \times F$ equipped with the metric $g_B\oplus f^2g_F$, where $g_B$ and $g_F$ denote the Riemannian metrics of $B$ and $F$, respectively, and $f : B \rightarrow \mathbb{R}$ is a positive function.

In \cite{MarcioRomildo}, the authors studied gradient Ricci solitons with warped product structure, and obtained all invariant solutions by translation of \eqref{eqGRS} when the base is locally conformally flat $\displaystyle \left(R^n,\bar{g}=\frac{1}{\varphi^2}g_0\right)$ and the fiber $F^m$ is an Einstein manifold. The solutions that were found excluded the case where $\varphi$ is constant, that is, when the base is the pseudo-Euclidean space. Inspired by the absence of this type of solution in \cite{MarcioRomildo}, in this work we study Riemannian warped products of the form $(\mathbb{R}^n\times_fF^m)$ satisfying equation \eqref{eqGRS}, where $f$ and $h$ are any differentiable functions. Here we do not use the hypothesis that the functions are invariant by translation.

The authors in \cite{Nazareno} showed that expanding or steady gradient Ricci solitons warped products whose torsion function reaches maximum and minimum must be Riemannian products. In the same work, they presented a complete example whose base is the Euclidean space, and that is generalized in this study.

In Theorem \ref{teoremafinvariante} we show that in a Riemannian warped product $\mathbb{R}^n\times_fF^m$, $n\geq 2$, with metric $g_0\oplus f^2g_F$ and $g_0$ denoting the Euclidean metric, the warping function $f$ is invariant by translation; that is,
$$f(x_1,\ldots,x_n)=P\left(\sum_{i=1}^{n}a_ix_i+b_i\right),$$
where $a_i,b_i\in\mathbb{R}$, and $P$ is a function of at least $C^1$. Already in Theorem \ref{teoremasistemadeEDO} we characterize the potential function of a gradient Ricci soliton when the torsion function depends only on one variable, which is a particular case of invariance by translation. From these two results we built several examples of gradient Ricci solitons. Among them, the examples \ref{exemplodoMárcio} and \ref{exemplodoNazareno} $i)$ are complete expanding gradient Ricci solitons.

\section{Statements}

In the following results, we consider
\begin{equation}\label{produtotorcido}
(M,g)=(\mathbb{R}^n\times_fF^m,g_0\oplus f^2g_F),
\end{equation}
a Riemannian warped product, where the base $(\mathbb{R}^n,g_0)$ is a Euclidean space with $n\geq 2$ and coordinates $(x_1,\ldots,x_n)$, the fiber $(F^m,g_F)$ is a Riemannian manifold with $m\geq 1$ and the warping function $f : \mathbb{R}^n \rightarrow R^*_+$ is a smooth function. We denote by $f,_{x_ix_j}$ and $h,_{x_ix_j}$ the second order derivatives of $f$ and $h$, respectively, with respect to $x_i$ and $x_j$. In addition, we identify fields and functions in $\mathbb{R}^n$ and $F^m$ with your lifting the $\mathbb{R}^n\times_fF^m$.

The first theorem classifies the warping function of a gradient Ricci soliton of the form \eqref{produtotorcido}. It tells us that this function is invariant under the $(n-1)$-dimensional translation group, which is a subgroup of the group of isometries from the Euclidean space $\mathbb{R}^n$.

\begin{theorem}\label{teoremafinvariante}
Let $(\mathbb{R}^n,g_0)$ be the Euclidean space, $n \geq 2$, with coordinates $x=(x_1,\cdots, x_n)$ and metric components $g_{0_{ij}}=\delta_{ij}$. Consider a warped product $M = \mathbb{R}^{n}\times _{f}F^{m}$ with metric $g = g_0 + f^{2}g_{F}$, where $F$ is a Riemannian manifold of dimension $m\geq 1$, $f, h:\mathbb{R}^{n}\rightarrow \mathbb{R}$ smooth functions and $f$ positive. If $(M,g)$ is a gradient Ricci soliton with potential function $h$, then the warping function $f$ is invariant by translation.
\end{theorem}

In the next result, we obtain the system of partial differential equations that describes warped products gradient Ricci solitons of the form $\mathbb{R}^n\times_fF^m$ when the potential function $h$ depends only on the base and the fiber $F^m$ is an Einstein manifold.

\begin{theorem}\label{teoremaEDP}
Let $(\mathbb{R}^n,g_0)$ be the Euclidean space, $n \geq 2$, with coordinates $x=(x_1,\cdots, x_n)$ and metric components $g_{0_{ij}}=\delta_{ij}$. Consider a warped product $M = \mathbb{R}^{n}\times _{f}F^{m}$ with metric $g = g_0 + f^{2}g_{F}$, where  $F$ is a Riemannian Einstein manifold with constant Ricci curvature $\lambda_{F}$, $m\geq 1$, $f, h:\mathbb{R}^{n}\rightarrow \mathbb{R}$ smooth functions and $f$ positive. Then $M$ is a gradient Ricci soliton with potential function $h$ if and only if the functions $f$ and $h$ satisfy:

\begin{equation}\label{sistemadeEDP}
\begin{cases}
fh,_{x_ix_j}-mf,_{x_ix_j}=0 , \ \forall i\neq j\\
fh,_{x_ix_i}-mf,_{x_ix_i}=\rho f, \ \forall i\\
\displaystyle \sum_{k=1}^{n}\left[-ff,_{x_kx_k}-(m-1)f,_{x_k}^2 + ff,_{x_k}h,_{x_k}\right]=\rho f^2-\lambda_F
\end{cases}.
\end{equation}
\end{theorem}

The following is an example of a complete solution for the system \eqref{sistemadeEDP}, considering the information given by Theorem \ref{teoremafinvariante}, i.e., taking $f$ to be invariant by translation.

\begin{example}\label{exemplodoMárcio}
Consider $\displaystyle \xi=\sum_{k=1}^{n}x_k$ and the functions
$$f(x_1,\ldots,x_n)= a_{1}e^{\sqrt{\frac{a}{m}}\xi} + a_{2}e^{-\sqrt{\frac{a}{m}}\xi},$$
and
$$h(x_1,\ldots,x_n) = \frac{-(n-1)a}{2}\sum_{k=1}^{n}x_{k}^{2} + a\sum_{k<l}x_{k}x_{l} + \sum_{k=1}^{n}c_{k}x_{k} + b,$$
where $a,a_1,a_2,b,c_k\in\mathbb{R}$, with $a>0$, $a_1^2+a_2^2\neq 0$, and $\displaystyle \sum_{k=1}^{n}c_{k} = 0$. So, we have
$$\begin{array}{lll}
f_{x_i}=a_{1}\sqrt{\frac{a}{m}}e^{\sqrt{\frac{a}{m}}\xi} - a_{2}\sqrt{\frac{a}{m}}e^{-\sqrt{\frac{a}{m}}\xi}, & & h,_{x_i}=-(n-1)ax_i+a\sum_{k\neq i}^{x_k}+c_i,\\
f,_{x_ix_i}=a_{1}\frac{a}{m}e^{\sqrt{\frac{a}{m}}\xi} + a_{2}\frac{a}{m}e^{-\sqrt{\frac{a}{m}}\xi}=\frac{a}{m}f, & & h,_{x_ix_i}=-(n-1)a,\\
f,_{x_ix_j}=\frac{a}{m}f, & & h,_{x_ix_j}=a.
\end{array}$$
With this information, it follows that the first equation of \eqref{sistemadeEDP} is automatically satisfied. From the second equation we get that $\rho=-na$. Now, replacing the data in the first member of the third equation of \eqref{sistemadeEDP}, we have
$$\begin{array}{l}
\displaystyle \sum_{k=1}^{n}\left[-\frac{a}{m}f^2-(m-1)\frac{a}{m}\left(a_1^2e^{2\sqrt{\frac{a}{m}}\xi}-2a_1a_2+a_2^2e^{-2\sqrt{\frac{a}{m}}\xi}\right)\right.\\
\left.+ff'\left(-(n-1)ax_k+a\sum_{l\neq k}x_l+c_k\right)\right]\\
\displaystyle =-\frac{na}{m}f^2-\frac{n(m-1)a}{m}(f^2-4a_1a_2)+ff'[-(n-1)a\xi+a\xi(n-1)]\\
\displaystyle =\rho f^2+4na\frac{(m-1)}{m}a_1a_2.
\end{array}$$
If $\displaystyle \lambda_F=-4a_1a_2na\frac{(m-1)}{m}=4a_1a_2\rho\frac{(m-1)}{m}$, we conclude that this expression is equal to $\rho f^2-\lambda_F$, and then the last equation of \eqref{sistemadeEDP} is satisfied.

As here the warping function is globally defined, if the fiber $F$ is a complete Einstein manifold, we have a complete expanding gradient Ricci soliton $\mathbb{R}^n\times_fF^m$ with $\rho=-na$, where the fiber has Ricci curvature $\displaystyle \lambda_F=4a_1a_2\rho\frac{(m-1)}{m}$.
\end{example}

Note that in the example above the Ricci curvature of the fiber can be negative, positive, or null, depending on whether $a_1$ and $a_2$ have the same sign, opposite signs, or if either of these constants is zero. Therefore, the fiber $F$ can be, for example, any $m$-dimensional model space $M^m(c)$, with sectional curvature $c\in\{-1,0,1\}$. In these cases, we have complete expanding gradient Ricci solitons $\mathbb{R}^n\times_fM^m(c)$. We also note that for \cite{Garcia-Rio} these solitons are not locally conformally flat.

Moreover, in the case that the Ricci curvature of $F$ is null with $a_1=0$ or $a_2=0$, we can choose for $F$, for example, the famous Schwarzschild metric, and get a beautiful example of an expanding gradient Ricci soliton.

In Theorem \ref{teoremaEDP}, we consider the case where the warping function depends on only one of the variables, which is a particular case of invariance by translation. Without loss of generality, we assume that $f$ depends only on $x_1$.

\begin{theorem}\label{teoremasistemadeEDO}
Let $(\mathbb{R}^n,g_0)$ be the Euclidean space, $n \geq 2$, with coordinates $x=(x_1,\cdots, x_n)$ and metric components $g_{0_{ij}}=\delta_{ij}$. Consider a warped product $M = \mathbb{R}^{n}\times _{f}F^{m}$ with metric $g = g_0 + f^{2}g_{F}$, where  $F$ is a Riemannian Einstein manifold with constant Ricci curvature $\lambda_{F}$, $m\geq 1$, $f, h:\mathbb{R}^{n}\rightarrow \mathbb{R}$ smooth functions and $f$ positive. If $f=f(x_1)$, then
$$h(x_1,\ldots,x_n)= h_1(x_1) + \sum_{k=1}^n\left(\frac{\rho}{2}x_k^2+a_kx_k + b_k\right),$$
and $h_1(x_1)$ is a smooth function satisfying
\begin{equation}
\begin{cases}\label{sistemadeEDO}
\lambda_F-ff''-(m-1)(f')^2+ff'h_1'=\rho f^2\\
fh_1''-mf''=\rho f
\end{cases},
\end{equation}
where $a_k,b_k\in\mathbb{R}$.
\end{theorem}

Next, we will outline some solutions for the system of equations \eqref{sistemadeEDO}.

\begin{example} \label{exemplodoNazareno}
Suppose that $h_1''(x_1)=h_1'(x_1)=0$ in \eqref{sistemadeEDO}. Thus, the potential function is given by
$$h(x_1,\ldots,x_n)=c + \sum_{k=2}^{n}\left(\frac{\rho}{2}x_k^2+a_kx_k+b_k\right),$$
with $a_k,b_k,c\in\mathbb{R}$.

From the second equation of \eqref{sistemadeEDO} we have
$$f''+\frac{\rho}{m}f=0,$$
and then $f$ is given by

$$f(x_1)=\begin{cases}
\begin{array}{ll}
c_1e^{\sqrt{-\frac{\rho}{m}}x_1}+c_2e^{-\sqrt{-\frac{\rho}{m}}x_1}, & \rho<0\\
c_1+c_2x_1, & \rho=0\\
c_1\sin\left(\sqrt{\frac{\rho}{m}}x_1\right)+c_2\cos\left(\sqrt{\frac{\rho}{m}}x_1\right), & \rho>0
\end{array}
\end{cases},$$

with $c_1,c_2\in\mathbb{R}$.

\begin{itemize}
\item[i)] If $f(x_1)=c_1e^{\sqrt{-\frac{\rho}{m}}x_1}+c_2e^{-\sqrt{-\frac{\rho}{m}}x_1}$, we have $(\mathbb{R}^n\times_fF^m,g_0\oplus f^2g_F)$ is an expanding gradient Ricci soliton, where the fiber $F$ has Ricci curvature $$\displaystyle \lambda_F=\frac{m-1}{m}(c_2^2-c_1^2)\rho.$$
If the fiber is a complete manifold, the gradient Ricci soliton is complete.

\item[ii)] If $f(x_1)=c_1+c_2x_1$, we get a steady gradient Ricci soliton, where the fiber $F$ has positive Ricci curvature $$\displaystyle \lambda_F=(m-1)c_2^2.$$

\item[iii)] If $f(x_1)=c_1\sin\left(\sqrt{\frac{\rho}{m}}x_1\right)+c_2\cos\left(\sqrt{\frac{\rho}{m}}x_1\right)$, we get a shrinking gradient Ricci soliton, where the fiber $F$ has positive Ricci curvature $$\displaystyle \lambda_F=\frac{m-1}{m}(c_1^2+c_2^2)\rho.$$
\end{itemize}
\end{example}

The previous example is a generalization of a complete example given in \cite{Nazareno}, where the authors obtained the function in $i)$ for $m>1$ and $\lambda_F\leq 0$, with the same potential function $h$.

In this example we also note that in the item $i)$ the Ricci curvature of the fiber can be negative (if $|c_2|>|c_1|$), null (if $c_1=c_2$) or positive (if $|c_1|>|c_2|$). Therefore, the fiber $F$ can be any $m$-dimensional model space  $M^m(c)$, with sectional curvature $c\in\{-1,0,1\}$. Then, we have a complete expanding gradient Ricci soliton $\mathbb{R}^n\times_fM^m(c)$. Here we also have that these solitons are not locally conformally flat, due to \cite{Garcia-Rio}.

\begin{example}
Consider the case where $m=1$. Therefore, $\lambda_F=0$, and the system \eqref{sistemadeEDO} becomes
$$\begin{cases}
fh_1''-f''=\rho f\\
-ff''+ff'h_1'=\rho f^2
\end{cases}$$
Multiplying the first equation by $f$ and comparing the two equations in the system, we have
$$\frac{h_1''}{h_1'}=\frac{f'}{f}.$$ 
So, we have $h_1'=kf$, with $k>0$. Substituting in the second equation, we conclude that $f$ must be solution of the ordinary differential equation
\begin{equation*}
f''-kff'-\rho f=0.
\end{equation*}
\end{example}

\section{Proof of Statements}

\begin{proof}[Proof of Theorem \ref{teoremafinvariante}:]
	
Let $(M,g)=(\mathbb{R}^n\times_fF^m,g_0\oplus f^2g_F)$ be a warped product manifold. Take $(p,q)\in\mathbb{R}^n\times_fF$ and suppose that $\{X_1,\ldots,X_n\}$ is an orthonormal base of $T_p(\mathbb{R}^n)$, i.e.,
$$g_0(X_i,X_j)=\delta_{ij},$$
and $\{Y_1,\ldots,Y_n\}$ is any base of $T_q(F)$.

As the base of the warped product is the Euclidean space $(\mathbb{R}^n,g_0)$, if $Hess_{g_0}f(X_i,X_j)=0$ for every pair of fields in the base,
$$f=\left(\sum_{i=1}^{n}a_ix_i+b_i\right),$$
whence the result follows.

Now, if there are at least a pair of fields $(X_i,X_j)$ such that $Hess_{g_0}f(X_i,X_j)\neq 0$, if $(M,g)$ is a gradient Ricci soliton, it follows from Theorem $1$ of \cite{PaulaRomildo} that $h$ depends only on the base and the fiber is an Einstein manifold.

It is known that in a warped product (see \cite{Oneill})
\begin{equation}
\begin{cases}\label{expressoesric}
\displaystyle  Ric_g(X_i,X_j)=Ric_{g_0}(X_i,X_j)-\frac{m}{f}Hess_{g_0}f(X_i,X_j)\\
Ric_g(X_i,Y_j)=0\\
\displaystyle  Ric_g(Y_i,Y_j)=Ric_{g_F}(Y_i,Y_j)-\left(\frac{\Delta_{g_0}f}{f}+\frac{||grad_{g_0}f||_{g_0}}{f^2}\right)g(Y_i,Y_j)
\end{cases}.
\end{equation}

Then, in coordinates the equation \eqref{eqGRS} of gradient Ricci solitons provides
\begin{equation} \label{eq1}
\begin{cases}
fh,_{x_ix_j}-mf,_{x_ix_j}=0, & i\neq j\\
fh,_{x_ix_i}-mf,_{x_ix_i}=\rho f & i=j
\end{cases}.
\end{equation}

Deriving the first equation from \eqref{eq1} with respect to $x_i$ and the second equation from \eqref{eq1} with respect to $x_j$, and comparing the results we get
\begin{equation} \label{eq4}
f,_{x_j}h,_{x_ix_i}-f,_{x_i}h,_{x_ix_j}=\rho f,_{x_j}.
\end{equation}

On the other hand, isolating $h,_{x_ix_i}$ and $h,_{x_ix_j}$ in the first and second equations of \eqref{eq1} respectively, and replacing \eqref{eq4}, we obtain
\begin{equation} \label{eq5}
\frac{m}{f}\left(f,_{x_j}f,_{x_ix_i}-f,_{x_i}f,_{x_ix_j}\right)=0.
\end{equation}

Without loss of generality, suppose that $f,_{x_i} \neq 0$ for all $i=1,\ldots,n$. Thus, it follows that \eqref{eq5} is equivalent to

$$\frac{f,_{x_ix_i}}{f,_{x_i}}=\frac{f,_{x_ix_j}}{f,_{x_j}}.$$

Then,

$$\log f,_{x_i}=\log f,_{x_j} + F_i(\hat{x_i}),$$

where $\hat{x_i}$ denotes that $F_i$ does not depend on the variable $x_i$. So

$$f,_{x_i}=e^{F_i(\hat{x_i})}f,_{x_j}.$$

Similarly,

$$f,_{x_j}=e^{F_j(\hat{x_j})}f,_{x_i}.$$

Substituting in the previous equation, we get

$$f,_{x_i}=e^{F_i(\hat{x_i})+F_j(\hat{x_j})}f,_{x_i}.$$

Therefore,

\begin{equation}\label{eqFiFj}
F_i(\hat{x_i})=-F_j(\hat{x_j}).
\end{equation}

Fixing $i$ and varying $j$, we will get that $F_i(\hat{x_i})$ is equal to a function that does not depend on $x_j$, for all $j\neq i$. Since $F_i$ is also not dependent on $x_i$, we conclude that $F_i$ is constant. Thus,

$$f,_{x_i}=c_{ij}f,_{x_j},$$

for all $i\neq j$. The characteristic of the equation above shows us that
$$f=P\left(\sum_{i=1}^{n}a_ix_i+b_i\right),$$
where $a_i,b_i\in\mathbb{R}$, $a_i=c_{ij}a_j$, and $P$ is a function at least $C^1$. Therefore, $f$ is invariant by translation.
\end{proof}

\vspace{.2in}

\begin{proof}[Proof of Theorem \ref{teoremaEDP}:]
Let $M = \mathbb{R}^{n}\times_{f} F^{m}$ be a warped product that is a gradient Ricci soliton with a potential function $h$. Suppose that $\{X_{1}, X_{2}, \ldots, X_{n}\}$ is an orthonormal base of $T_{p}(\mathbb{R}^{n})$ and $\{Y_{1}, Y_{2}, \ldots, Y_{m}\}$ is any base of $T_{q}(F)$. Then $g_{0}(X_{i}, X_{j}) = \delta_{ij}$ and denoting  the Ricci tensor of manifolds $M, \mathbb{R}^{n}$, and $F$ by $Ric_g, Ric_{g_0}, Ric_{g_F}$, respectively, we have that the expressions for the Ricci tensor are given by \eqref{expressoesric}.

As the base of the warped product is the Euclidean space $(\mathbb{R}^n,g_0)$ and the fiber is an Einstein manifold, it follows that
\begin{eqnarray}\label{ric2}
\left\{                                                                                             \begin{array}{lcl}
Ric_{g_0}(X_{i}, X_{j}) &=& 0\\
Hess_{g_0}f(X_{i}, X_{j}) &=& f_{,x_{i}x_{j}}\\
\Delta_{g_0}f &=&\displaystyle\sum_{k=1}^{n}f_{,x_{k}x_{k}}\\
\|grad_{g_0}f\|_{0} &=& \displaystyle\sum_{k=1}^{n}f_{,x_k}^{2}\\
Ric_{g_F}(Y_{i}, Y_{j}) &=& \lambda_{F} g_{F}(Y_{i}, Y_{j})\\
Hess_{g}(h)(X_{i}, X_{j}) &=& h_{,x_{i}x_{j}}\\
Hess_g(h)(X_{i}, Y_{j}) &=& 0\\
Hess_g(h)(Y_{i}, Y_{j}) &=& \sum_{k=1}^{n}ff,_{x_{k}}h,_{x_{k}}g_{F}(Y_{i},Y_{j})

\end{array}
\right.
\end{eqnarray} 	

Furthermore, as we are assuming that $(M,g)$ is a gradient Ricci soliton with potential function $h: \mathbb{R}^n\longrightarrow\mathbb{R}$, we have
\begin{equation}\label{GRS}
Ric_g + Hess_g(h) = \rho g.
\end{equation}
Replacing \eqref{ric2} in \eqref{GRS}, we get
\begin{equation*}
\begin{cases}
fh,_{x_ix_j}-mf,_{x_ix_j}=0 , \ \forall i\neq j\\
fh,_{x_ix_i}-mf,_{x_ix_i}=\rho f, \ \forall i\\
\displaystyle \sum_{k=1}^{n}\left[-ff,_{x_kx_k}-(m-1)f,_{x_k}^2 + ff,_{x_k}h,_{x_k}\right]=\rho f^2-\lambda_F
\end{cases}.
\end{equation*}
\end{proof}

\vspace{.2in}

\begin{proof}[Proof of Theorem \ref{teoremasistemadeEDO}:]
Suppose that $f=f(x_1)$. Then, from the first equation of \eqref{sistemadeEDP} we get that
$$h,_{x_ix_j}=0;$$
Thus, it is evident that, $h$ is a function of separable variables, such as
$$h(x_1,\ldots,x_n)=\sum_{k=1}^{n}h_k(x_k).$$

In the second equation of \eqref{sistemadeEDP}, if $i=1$ we get the second equation of \eqref{sistemadeEDO}; if $i\neq 1$ we have
$$h,_{x_ix_i}=\rho,$$
whence we conclude that
$$h_i(x_i)=\frac{\rho}{2}\left(x_i^2+a_ix_i+b_i\right),$$

for all $i\neq 1$. Finally, the third equation of \eqref{sistemadeEDP} follows from the first equation of \eqref{sistemadeEDO}.
\end{proof}

\bibliographystyle{acm}

\bibliography{bibliography}

\end{document}